\documentclass{amsart}

\usepackage{amsthm,amssymb,amsmath,a4wide,graphicx,tikz}
\usepackage[english]{babel}
\usepackage{subfigure}
\pdfoutput=1

\begin{document}

\title{Transformations generating negative $\beta$-expansions}
\author{Karma Dajani}
\author{Charlene Kalle$^\dagger$}
\address{Department of Mathematics\\
Utrecht University\\
Postbus 80.000\\
3508 TA Utrecht\\
the Netherlands} \email{k.dajani1@uu.nl}
\address{Mathematics Institute \\
University of Warwick \\
Coventry CV4 7AL\\
 United Kingdom}
 \email{c.kalle@warwick.ac.uk }

\thanks{{\scriptsize $^\dagger$This author was supported by the EU FP6 Marie Curie Research Training Network CODY (MRTN 2006 035651).}}


\maketitle

\begin{abstract}
We introduce a family of dynamical systems that  generate negative $\beta$-expansions and study the support of the invariant measure which is absolutely continuous with respect to Lebesgue measure. We give a characterization of the set of digit sequences that is produced by a typical member of this family of transformations. We discuss the meaning of greedy expansions in the negative sense, and show that there is no transformation in the introduced family of dynamical systems that generates negative greedy. However, if one looks at random algorithms, then it is possible to define a greedy expansion in base $-\beta$.
\end{abstract}

\newtheorem{prop}{Proposition}[section]
\newtheorem{theorem}{Theorem}[section]
\newtheorem{lemma}{Lemma}[section]
\newtheorem{cor}{Corollary}[section]
\newtheorem{remark}{Remark}[section]
\theoremstyle{definition}
\newtheorem{defn}{Definition}[section]
\newtheorem{ex}{Example}[section]

\maketitle

\section{Introduction}

Given a real number $\beta>1$, it is well known that we can write every $x$ in the unit interval as
\begin{equation}\label{q:exp}
x = \sum_{k=1}^{\infty} \frac{b_k}{\beta^k},
\end{equation}
where the $b_k$'s are all taken from the set of integers $\{0,1, \ldots, \lfloor \beta \rfloor\}$. Here $\lfloor \beta \rfloor$ is the largest integer not exceeding $\beta$. The expression (\ref{q:exp}) is called a $\beta$-expansion of $x$ with digits in $\{0,1, \ldots, \lfloor \beta \rfloor\}$ and the sequence $b_1 b_2 \cdots$ is called a digit sequence for $x$. One way to generate such expansions is by iterating the map $ x \mapsto \beta x \, (\mathrm{mod}\, 1)$. The expansions given by this map are the greedy $\beta$-expansions, in the sense that if $b_1, b_2, \ldots, b_{n-1}$ are already known, then $b_n$ is the largest element from the set $\{ 0,1, \ldots, \lfloor \beta \rfloor \}$, such that $\sum_{k=1}^n \frac{b_k}{\beta^k} \le x$.
In \cite{IS09}, Ito and Sadahiro studied a dynamical system that can be used to generate $\beta$-expansions with negative bases. For each real number $\beta>1$, they defined a transformation that generates for each $x$ in some interval an expression of the form
\begin{equation}\label{q:negexpcl}
x = \sum_{k=1}^{\infty} \frac{b_k}{(-\beta)^k} = \sum_{k=1}^{\infty} (-1)^k \frac{b_k}{\beta^k},
\end{equation}
where the digits $b_k$ are again in the set $\{ 0,1, \ldots , \lfloor \beta \rfloor\}$. Their dynamical system generates what they call `greedy expansions in negative base' and is defined on the interval $\big[ \frac{-\beta}{\beta+1}, \frac{1}{\beta+1} \big]$ as follows.
\begin{equation}\label{q:is}
 T x = \left\{
\begin{array}{ll}
-\beta x-\lfloor \beta \rfloor,& \text{if } \displaystyle \frac{-\beta}{\beta+1} \le x \le \frac{1}{\beta +1} - \frac{\lfloor \beta \rfloor}{\beta},\\
\\
-\beta x-j, & \text{if } \displaystyle \frac{1}{\beta+1}-\frac{j+1}{\beta} < x \le \frac{1}{\beta +1}-\frac{j}{\beta}, \quad j \in \{0, 1, \ldots, \lfloor \beta \rfloor -1\}.
\end{array}
\right.\end{equation}
We call expressions of the form (\ref{q:negexpcl}) {\em negative $\beta$-expansions} with digits in $\{ 0,1, \ldots , \lfloor \beta \rfloor\}$. In \cite{FL09}, Frougny and Lai explored the properties of the expansions generated by this transformation and made a further comparison with the $\beta$-expansions as given in (\ref{q:exp}).
\vskip .2cm
In this paper we have a closer look at the dynamics behind negative $\beta$-expansions. For simplicity of the exposition, we only look at the two digit situation, but most of the results are easily generalized to more digits. In Section~\ref{s:trfm} we introduce a family of dynamical systems that generate negative $\beta$-expansions by iterations, and study the support of the invariant measure which is absolutely continuous with respect to Lebesgue measure. In Section~\ref{s:order} we give a characterization of the set of digit sequences that is produced by a typical member of this family of transformations. We discuss the meaning of greedy expansions in the negative sense and show that there is no transformation in the introduced family of dynamical systems that generates negative greedy $\beta$-expansions. However, if one looks at random algorithms, then it is possible to define a greedy expansion in base $-\beta$. This is done in Section~\ref{s:random}, where we also have a look at unique expansions.

\section{Being negative}\label{s:trfm}

Let $\beta >1$ be a real number and consider expansions of the form (\ref{q:negexpcl}) with $b_k \in \{0,1\}$ for each $k \ge 1$. Since all the even $k$'s contribute a non-negative value to the total sum and all the odd $k$'s a non-positive value, the smallest number we can obtain is when $b_k=0$ if $k$ is even and $b_k =1$ if $k$ is odd. Similarly, we get the largest number when $b_k=1$ if $k$ is even and $b_k=0$ for odd values $k$. This gives
\[ M^- = - \sum_{k=1}^{\infty} \frac{1}{\beta^{2k-1}} = \frac{- \beta}{\beta^2-1} \quad \mathrm{and} \quad M^+ = \sum_{k=1}^{\infty} \frac{1}{\beta^{2k}} = \frac{1}{\beta^2 -1}.\]
Hence, every number with an expression of the form (\ref{q:negexpcl}) with $b_k \in \{0,1\}$ for all $k \ge 1$, is an element of the interval $[M^-, M^+]$. We have the following useful proposition.
\begin{lemma}\label{l:firstdigit}
Let $x \in [M^-,M^+]$ and suppose $x$ has the negative $\beta$-expansion
\[ x  = \sum_{k=1}^{\infty} (-1)^k \frac{b_k}{\beta^k},\]
with $b_k \in \{0,1\}$ for all $k \ge 1$.
\begin{itemize}
\item[(i)] if $b_1 = 0$, then $x \in \big[ -\frac{1}{\beta (\beta^2-1)}, M^+\big] $,
\item[(ii)] if $b_1 =1$, then $x \in \big[M^-, \frac{1}{\beta^2-1}-\frac{1}{\beta} \big]$.
\end{itemize}
\end{lemma}

\begin{proof}
(i) Suppose $b_1 =0$. Then the minimal value of the expression $\sum_{k=2}^{\infty} (-1)^k \frac{b_k}{\beta^k}$, is achieved if $b_n = 1$ for all odd $n \ge 3$ and $b_n=0$ for all even values of $n$. This gives
\[x \ge \sum_{k=1}^{\infty} (-1)^k \frac{1}{\beta^{2k+1}} = -\frac{1}{\beta^3} \frac{1}{1-1/\beta^2} = -\frac{1}{\beta(\beta^2-1)}.\]
(ii) If $b_1=1$, then the maximal value of $ \sum_{k=2}^{\infty} (-1)^k \frac{b_k}{\beta^k},$ is achieved if $b_n=0$ for all odd values of $n$ and $b_n=1$ for all even values of $n$. Hence,
\[ x \le -\frac{1}{\beta} + \sum_{k=1}^{\infty} \frac{1}{\beta^{2k}} = -\frac{1}{\beta} + \frac{1}{\beta^2-1}. \qedhere\]
\end{proof}

\subsection{Conditions for transformations}
We would like a family of transformations that generate negative $\beta$-expansions with digits in $\{0,1\}$. Therefore, consider the maps $T_j x = -\beta x - j$ for $j \in \{0,1\}$. The family of transformations that we will introduce, use the map $T_0$ on a subinterval of $[M^-,M^+]$ of the form $[\alpha, M^+]$ and $T_1$ on the complement $[M^-, \alpha)$. If we want to iterate such a transformation, then this combination of $T_0$ and $T_1$ needs to map the interval $[M^-, M^+]$ into itself. Note that $T_0 \big[ -\frac{1}{\beta (\beta^2-1)}, M^+ \big] = [M^-, M^-]$ and $T_1 \big[M^-, \frac{1}{\beta^2-1}-\frac{1}{\beta} \big] = [M^-, M^+]$. We can construct a transformation according to the description above if for each $x \in [M^-,M^+]$, either $T_0x \in [M^-,M^+]$ or $T_1 x \in [M^-,M^+]$. Thus, only if the interval $\big[ \frac{1}{\beta^2-1} - \frac{1}{\beta}, -\frac{1}{\beta(\beta^2-1)} \big]$ is non-empty, which happens if and only if $1< \beta \le 2$. This divides $[M^-,M^+]$ into three parts:
\begin{equation}\label{q:switch}
U_1=\Big[M^-, \frac{1}{\beta^2-1}-\frac{1}{\beta} \Big), \quad S=\Big[ \frac{1}{\beta^2-1}-\frac{1}{\beta}, -\frac{1}{\beta(\beta^2-1)}  \Big], \quad U_0 = \Big( -\frac{1}{\beta(\beta^2-1)}, M^+\Big].
\end{equation}
Then $[M^-,M^+]=U_1 \cup S \cup U_0$, where this union in disjoint. On $U_1$ we need to use $T_1$ and on $U_0$ we use $T_0$. Therefore, $U_1$ and $U_0$ are called {\em uniqueness regions}. On $S$ we have a choice between $T_0$ and $T_1$ and this interval is called a {\em switch region}. See Figure~\ref{f:is}(a).

\begin{prop}
Every $x \in [M^-,M^+]$ has an expansion of the form (\ref{q:negexpcl}) with $b_k \in \{0,1\}$ for all $k \ge 1$ if and only if $1 < \beta \le 2$.
\end{prop}

\begin{proof}
By Lemma~\ref{l:firstdigit} we know that all $x \in [M^-,M^+]$ have expansions of the form (\ref{q:negexpcl}) iff $-\frac{1}{\beta(\beta^2-1)} \le \frac{1}{\beta^2-1}-\frac{1}{\beta}$ and this holds iff $\beta \le 2$.
\end{proof}

Suppose that $1 < \beta \le 2$ and let $S$ be as in (\ref{q:switch}). Then for each $\alpha \in S$, define two transformations $L= L_{\beta, \alpha}: [M^-, M^+] \to [M^-,M^+]$ and $R= R_{\beta, \alpha}: [M^-, M^+] \to [M^-,M^+]$ by setting
\[ L\, x = \left\{
\begin{array}{ll}
-\beta x-1,& \text{if } x \le \alpha,\\
-\beta x, & \text{if }  x > \alpha,
\end{array}
\right. \ \text{and } \ R\, x = \left\{
\begin{array}{ll}
-\beta x-1,& \text{if } x < \alpha,\\
-\beta x, & \text{if } x \ge \alpha.
\end{array}
\right. \]

We can define for each $x \in [M^-,M^+]$ the digit sequence $b(x)=b_1(x) b_2(x) \cdots$ given by $R$ by setting for $n \ge 1$,
\[ b_n=b_n(x) = \left\{
\begin{array}{ll}
0,& \text{if } R^{n-1}x \ge \alpha,\\
1, & \text{if }  R^{n-1} x < \alpha.
\end{array}
\right. \]
Then for each $n\ge 1$,
\[ x = \sum_{k=1}^n (-1)^k \frac{b_k}{\beta^k} + (-1)^n \frac{R^n x}{\beta^n}. \]
Since $R^n x \in [M^-,M^+]$ for each $n \ge 1$, this converges and thus, we can write $x = \sum_{k=1}^{\infty} (-1)^k \frac{b_k}{\beta^k}$. Hence, for each $1 <\beta \le 2$ and each choice of $\alpha \in S$, we get a transformation $R_{\beta, \alpha}$ that generates expansions of the form (\ref{q:negexpcl}) with $b_k \in \{0,1\}$, and $x \in [M^-,M^+]$. We give an example.

\begin{remark}{\rm
(i) Note that the transformations $R$ and $L$ only differ at the point $\alpha$. We study the transformation $R$ only, since for any $\alpha$, the transformation $L=L_{\beta, \alpha}$ is isomorphic to the transformation $R_{\beta, \tilde \alpha}$, where $\tilde \alpha = -\frac{1}{\beta+1}-\alpha$. The isomorphism $\theta :[M^-,M^+] \to [M^-,M^+]$ is given by $\theta (x) = -\frac{1}{\beta+1}-x$.\\
(ii) If $\beta=2$, then the switch region $S$ consists of the single point $-\frac{1}{\beta(\beta^2-1)} = \frac{1}{\beta^2-1}-\frac{1}{\beta}$. Then, the maps $L$ and $R$ are both isomorphic to the full one-sided uniform Bernoulli shift on two symbols. Since the same holds for the doubling map $x \mapsto 2 x$ (mod 1), in this case the maps $L$ and $R$ are also both isomorphic to the doubling map. Therefore, we will not consider $\beta=2$ further.
}\end{remark}

\begin{ex}
For two digits, the transformation $T$ studied in \cite{IS09} by Ito and Sadahiro (see (\ref{q:is})) is obtained by taking $L_{\beta, \alpha}$ with $\alpha = \frac{1}{\beta+1}-\frac{1}{\beta}$. We see this map in Figure~\ref{f:is}. Note that the interval $\big[ -\frac{\beta}{\beta+1} , \frac{1}{\beta+1}\big]$ is an attractor, which can be seen from Figure~\ref{f:is}(b).
\begin{figure}[ht]
\centering
\subfigure[$T_0$ and $T_1$]{\includegraphics{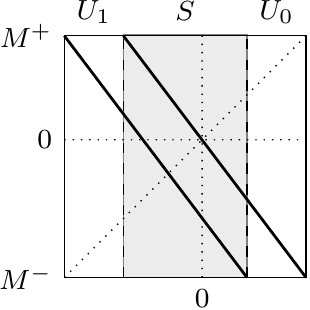}}
\quad
\subfigure[$\alpha = \frac{1}{\beta+1} - \frac{1}{\beta}$]{\includegraphics{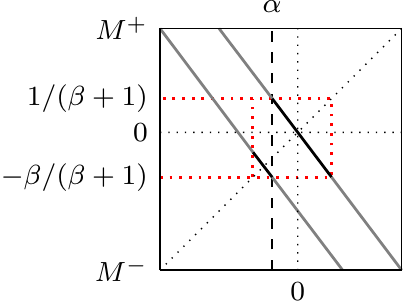}}
\quad
\subfigure[$R_{\beta, \alpha}$ from (b) in the red box]{\includegraphics{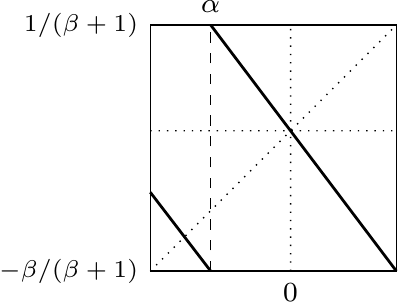}}
\caption{In (a) we see the full maps $T_0$ and $T_1$ and in (b) we see the map $R_{\beta, \alpha}$ with $\alpha=\frac{1}{\beta+1}-\frac{1}{\beta}$, the choice from \cite{IS09}. In (c) we see this transformation on the interval $\big[-\frac{\beta}{\beta+1}, \frac{1}{\beta+1}\big]$.}
\label{f:is}
\end{figure}
\end{ex}

\subsection{Attract and support}

For two digits $\{0,1\}$ any transformation $R=R_{\beta, \alpha}$ with $1< \beta < 2$ and $\alpha \in S$ has exactly one point of discontinuity. By results from Li and Yorke (\cite{LY78}), there is a unique invariant probability measure absolutely continuous with respect to Lebesgue ({\em acim}). From the same results, it follows immediately that this measure is ergodic and that the support of the acim is a forward invariant set, which contains an interval that has $\alpha$ as an interior point. It remains to determine what the support of the acim is.

We can easily identify such a forward invariant set, by using the images of $\alpha$ under $T_0$ and $T_1$. Note that the middle of the interval $S$ is the point $- \frac{1}{2(\beta+1)}$. By symmetry it is enough to consider $\alpha \le -\frac{1}{2(\beta+1)}$.

First suppose that $\alpha \le -\frac{1}{\beta(\beta+1)}$, see Figure~\ref{f:attract}(a). Then $\beta^2 \alpha \le -\beta \alpha -1$ and $-\beta^3 \alpha -1 \le -\beta \alpha$. Consider the interval $[\beta^2 \alpha, -\beta \alpha]$. Then
\[ R [\beta^2 \alpha, -\beta \alpha] \subseteq [-\beta \alpha -1, -\beta^3 \alpha -1] \cup [\beta^2 \alpha, -\beta \alpha] \subseteq [\beta^2 \alpha, -\beta \alpha].\]
Thus, the interval $[\beta^2 \alpha, -\beta \alpha]$ is forward invariant with $\alpha$ in its interior, which implies that it contains the support of the acim.

If $\alpha > -\frac{1}{\beta(\beta+1)} $, then $ -\beta \alpha -1 < \beta^2 \alpha$. See Figure~\ref{f:attract}(b). Consider the interval $[-\beta \alpha -1, -\beta \alpha]$. Then,
\[ R[-\beta \alpha -1, -\beta \alpha] \subseteq [-\beta \alpha -1, \beta^2 \alpha +\beta -1] \cup [\beta^2 \alpha, -\beta \alpha].\]
Hence, in this case the interval $[-\beta \alpha -1, -\beta \alpha]$ contains the support of the acim.

In case $\alpha > -\frac{1}{2(\beta+1)}$, for $\alpha \le -\frac{\beta-1}{\beta(\beta+1)}$, the invariant set is $[-\beta \alpha -1, -\beta \alpha]$ and for $\alpha > -\frac{\beta - 1}{\beta(\beta+1)}$, the invariant set is $[-\beta \alpha -1, \beta^2 \alpha + \beta -1]$.
\begin{figure}[ht]
\centering
\subfigure[$\alpha < -\frac{1}{\beta(\beta+1)}$]{\includegraphics{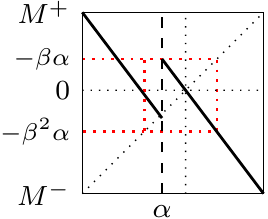}
\quad
\includegraphics{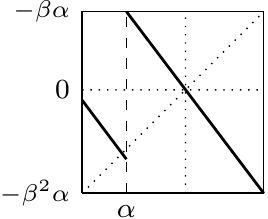}}
\hspace{1cm}
\subfigure[$-\frac{1}{\beta(\beta+1)} < \alpha < -\frac{1}{2(\beta+1)}$]{\includegraphics{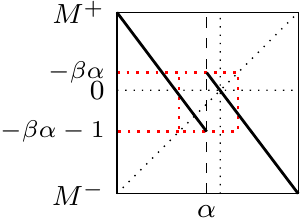}
\quad
\includegraphics{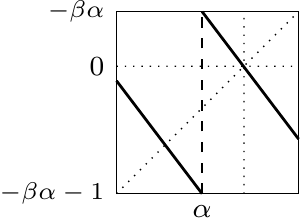}}
\caption{Two choices of $\alpha$ for the same $\beta$ that give different forward invariant sets. For both (a) and (b), the map on the right is the map on the left restricted to this forward invariant set.}
\label{f:attract}
\end{figure}

We consider an example in which we can identify the support.
\begin{ex}
Let $\alpha$ be one of the two endpoints of $S$, so $\alpha = \frac{1}{\beta^2-1} - \frac{1}{\beta}$, or $\alpha = -\frac{1}{\beta(\beta^2-1)}$, and let $R=R_{\beta, \alpha}$. See Figure~\ref{f:support} for examples with $\alpha = -\frac{1}{\beta(\beta^2-1)}$.

To identify the support of these transformations, by symmetry it is enough to consider only one of the two. Take  $\alpha = -\frac{1}{\beta(\beta^2-1)}$. The fixed points of $T_0$ and $T_1$ are important. For $T_0$ the fixed point is $0$ and for $T_1$ this is $-\frac{1}{\beta+1}$. \\
(i) First assume that $-\beta \alpha -1 >0$. See Figure~\ref{f:support} (a). Then the set $[M^-,\beta^2 \alpha + \beta] \cup [-\beta \alpha -1, M^+]$ is forward invariant. Moreover, if we take an interval $[a,b] \subseteq [M^-, M^+]$ with $\alpha \in (a,b)$ and such that $b \le \beta^2 \alpha + \beta$, then for $n \ge 1$ small enough,
\[ R^n(\alpha, b) = \Big( R^n b, \frac{1}{\beta^2-1} \Big) \, \text{ for odd $n$ and } \, R^n(\alpha, b) = \Big( -\frac{\beta}{\beta^2-1}, R^n b \Big) \, \text{ for even $n$}. \]
Since $R$ is expanding, the Lebesgue measure of this interval grows with a factor $\beta$ with each iteration. Hence, after some $n$, $[M^-, \alpha] \subseteq R^n(\alpha, b)$. This implies that
\[ R^{n+2} (a,b) = [M^-,\beta^2 \alpha + \beta] \cup [-\beta \alpha -1, M^+].\]
Hence, the support of the acim of $R$ is exactly the set $[M^-,\beta^2 \alpha + \beta] \cup [-\beta \alpha -1, M^+]$, i.e., the union of two disjoint intervals.\\
(ii) Now, assume that $-\beta \alpha -1 \le 0$. See Figure~\ref{f:support} (b). Then, for any interval $[a,b] \subseteq [M^-, M^+]$ with $\alpha \in (a,b)$, there is an $n\ge 1$, such that $[M^-, \alpha] \subseteq R^n (\alpha, b)$. Since $-\beta \alpha -1 \le 0$, we have that
\[ [M^-, \alpha] \cup R[M^-, \alpha] \cup R^2[M^-, \alpha] = [M^-, M^+].\]
Hence, the acim in this case is fully supported.
\vskip .1cm
If $\alpha =  \frac{1}{\beta^2-1} - \frac{1}{\beta}$, then for $-\beta \alpha < -\frac{1}{\beta+1}$ the support is $[M^-, -\beta \alpha] \cup [\beta^2 \alpha -1, M^+]$ and if $-\beta \alpha \ge -\frac{1}{\beta+1}$, then the support is the whole interval $[M^-,M^+]$.
\begin{figure}[ht]
\centering
\subfigure[$-\beta \alpha -1 >0$]{
\includegraphics{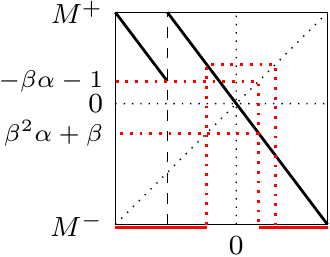}}
\hspace{.2cm}
\subfigure[$-\beta \alpha -1 \le 0$]{
\includegraphics{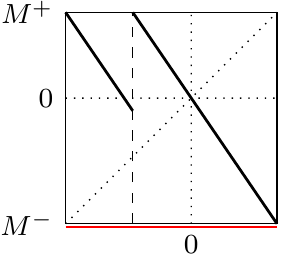}}
\hspace{.2cm}
\subfigure[A map from Example~\ref{x:three}]{\includegraphics{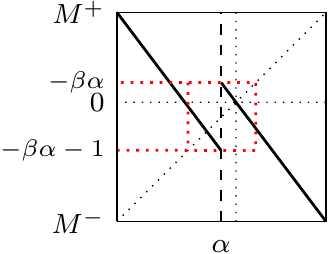}
\includegraphics{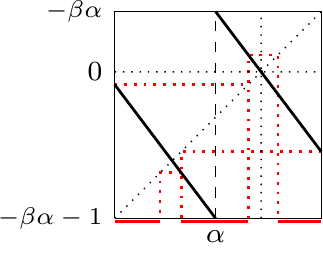}}
\caption{In (a) and (b) we see two cases for the support of the acim of $R_{\beta, \alpha}$ with $\alpha = -\frac{1}{\beta(\beta^2-1)}$. In (c) is an example of a map of which the support consists of at least three disjoint intervals. On the left hand side is the complete picture and we see that the interesting dynamics happens in the red box. On the right hand side we see the map in this red box and we can identify the three intervals.}
\label{f:support}
\end{figure}
\end{ex}

In general, for arbitrary choices of $\alpha$ the support is always a union of closed, disjoint intervals, but many things can happen. We give an example where the number of intervals is at least three. Liao and Steiner (\cite{LS10})
have explicit constructions of examples of transformations of which the support of the acim is a union of more than three intervals. To be more precise, they gave examples of acims of which the support is the union of a number of intervals from the sequence $1,2,5,10,21,22,45,46,\ldots$.

\begin{ex}\label{x:three}
Take $\beta$ such that $\beta^3-\beta -1=0$ and let $\alpha \in \big( -\frac{1}{\beta^2 (\beta+1)}, -\frac{\beta-1}{\beta^2} \big)$. Then
\begin{equation}\label{q:sup}
-\frac{1}{\beta+1} < \beta^2 \alpha < \alpha, \quad \text{and} \quad \alpha < \beta^2 \alpha + \beta -1 < 0.
\end{equation}
Define the set
\[ [-\beta \alpha-1, -\beta^3 \alpha -1] \cup [\beta^2 \alpha, \beta^2 \alpha + \beta -1] \cup [-\beta^3 \alpha -\beta^2 + \beta, -\beta \alpha].\]
Then, by (\ref{q:sup}), this is a forward invariant set. Moreover, it contains $\alpha$ in its interior. Thus, the support of the acim must be contained in this set. Since for any interval containing $\alpha$ in its interior, the forward image has nonempty intersections with both of the other two intervals, the support of the acim intersects all three of these intervals. See Figure~\ref{f:support}(c) for an example.
\end{ex}

\subsection{Invariant Density}

Invariant densities for piecewise linear increasing maps have been thoroughly investigated (see e.g., \cite{Gora09},  \cite{Kopf90}). We use a trick by Hofbauer (\cite{Hof81}) to view our map $R=R_{\beta,\alpha}$ as a factor of a piecewise linear and increasing map $T=T_{\beta,\alpha}$. This allows one to derive the invariant density for the $R$ map using the invariant density for the $T$ map. To do this, we first view $R$ as a map on $\big[0,\frac{1}{\beta-1} \big]$ as follows. Let
$\phi:\big[\frac{-\beta}{\beta^2-1},\frac{1}{\beta^2-1}\big]\to \big[0,\frac{1}{\beta-1}\big]$ be given by
\[\phi(x)=x+\displaystyle\frac{\beta}{\beta^2-1}.\]

\noindent Define $W=W_{\beta,\alpha}:\big[0,\frac{1}{\beta-1}\big]\to \big[0,\frac{1}{\beta-1}\big]$ by
\[W(x)=\phi\circ R\circ \phi^{-1}(x)
\left\{
\begin{array}{ll}
-\beta x+\displaystyle\frac{1}{\beta-1},& \text{if }  x\in \Big[0,\alpha+\displaystyle\frac{\beta}{\beta^2-1}\Big],\\
\\
-\beta x+\displaystyle\frac{\beta}{\beta-1}, & \text{if } x\in \Big(\alpha+\displaystyle\frac{\beta}{\beta^2-1},\displaystyle\frac{1}{\beta-1} \Big].
\end{array}
\right.\]

\noindent Define $T=T_{\beta,\alpha}:\big[0,\frac{2}{\beta-1}\big]\to \big[0,\frac{2}{\beta-1}\big]$ by
\[ T (x) = \left\{
\begin{array}{ll}
\displaystyle\frac{2}{\beta-1}-W(x),& \text{if }  x\in \Big[0, \displaystyle \frac{1}{\beta-1}\Big],\\
\\
W\Big(\displaystyle\frac{2}{\beta-1}-x\Big), & \text{if } x\in \Big(\displaystyle\frac{1}{\beta-1},\displaystyle\frac{2}{\beta-1} \Big].
\end{array}
\right.\]

\noindent We see these maps in Figure~\ref{f:w}.
\begin{figure}[ht]
\centering
\subfigure[The map $R_{\beta, \alpha}$ from Figure~\ref{f:support}(c)]{\includegraphics{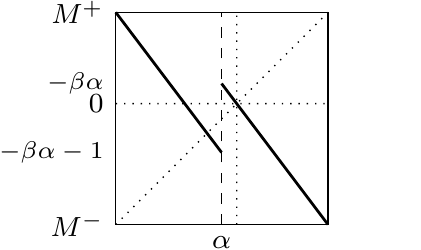}}
\quad
\subfigure[The maps $T$ and $W$ for $R_{\beta, \alpha}$]{\includegraphics{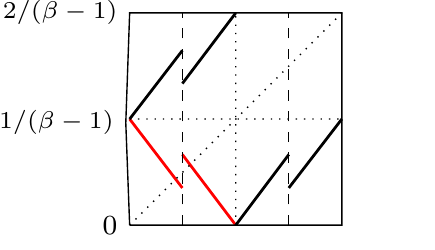}}
\caption{The maps $T$ and $W$ for a map $R_{\beta, \alpha}$. In (b), the red lines indicate the map $T$ and the black lines the map $W$.}
\label{f:w}
\end{figure}

\noindent Finally, define the map $\tau:\big[0,\frac{2}{\beta-1}\big]\to \big[0,\frac{1}{\beta-1}\big]$ by
\[ \tau (x) = \left\{
\begin{array}{ll}
x,& \text{if }  x\in \Big[0,\displaystyle\frac{1}{\beta-1}\Big],\\
\\
\displaystyle\frac{2}{\beta-1}-x, & \text{if } x\in \Big(\displaystyle\frac{1}{\beta-1},\displaystyle\frac{2}{\beta-1} \Big].
\end{array}
\right.\]

\noindent Then $W\circ \tau=\tau \circ T$, and it is easily seen that $\tau$ is a factor map. Notice that $\tau$ is 2-to-1 map, and that
the map $T$ is symmetric around the origin with
\[T\left( \Big(0,\frac{1}{\beta-1} \Big)\right)=\Big(\frac{1}{\beta-1},\frac{2}{\beta-1}\Big) \quad
\text{ and }\quad T\left( \Big(\frac{1}{\beta-1},\frac{2}{\beta-1} \Big)\right)=\Big(0,\frac{1}{\beta-1}\Big).\]
Thus if $h$ is the invariant density for $T$, then the invariant density for $W$ is given by
\[ g(x)=h(x)+h\Big(\frac{2}{\beta-1}-x\Big)=2h(x).\]
From this it follows that the non-normalized invariant density for the $R$ map is given by
\[k(x)=g\Big(x+\frac{\beta}{\beta^2-1}\Big)=2h\Big(x+\frac{\beta}{\beta^2-1}\Big),\]
for $x\in \big[\frac{-\beta}{\beta^2-1},\frac{1}{\beta^2-1}\big]$.

\section{Orderings}\label{s:order}
Let $1< \beta < 2$ and take $\alpha \in S$. Let $R=R_{\beta, \alpha}$ be the corresponding negative $\beta$-transformation. We can give a characterization of the digit sequences generated by $R$. For $\{0, 1\}^{\mathbb N}$, define the ordering $\prec$, which is called the {\em alternate ordering} as follows. We say that $b=b_1 b_2 \cdots \prec d=d_1 d_2 \cdots$ if and only if there is an $n \ge 1$, such that $b_k = d_k$ for all $1 \le k \le n-1$ and $(-1)^n (b_n-d_n) <0$. Then $b \preceq d$ if and only if $b=d$ or $b \prec d$. We can define $\succ$ and $\succeq$ similarly.
\begin{lemma}\label{l:preserveorder}
Let $x,y \in [M^- , M^+]$, and let $b(x), b(y)$ be the corresponding digit sequences generated by $R$. Then $x < y$ if and only if $b(x) \prec b(y)$.
\end{lemma}

\begin{proof}
Suppose $x < y$, then $b(x) \neq b(y)$. Let $n\ge 0$ be the first index such that $b_{n+1}(x) \neq b_{n+1}(y)$, then \begin{eqnarray*}
x &=& \sum_{k=1}^n (-1)^k \frac{b_k(x)}{\beta^k} + (-1)^n \frac{R^n x }{\beta^n} =  \sum_{k=1}^n (-1)^k \frac{b_k(y)}{\beta^k} + (-1)^n \frac{R^n x }{\beta^n}\\
&<&  \sum_{k=1}^n (-1)^k \frac{b_k(y)}{\beta^k} + (-1)^n \frac{R^n y }{\beta^n}=y.\\
\end{eqnarray*}
This implies that $(-1)^nR^n x<(-1)^nR^n y$. If $n$ is even, then $R^n x<R^n y$ implying that $b_{n+1}(y)<b_{n+1}(x)$. If $n$ is odd, then
$R^n y<R^n x$, so $b_{n+1}(x)<b_{n+1}(y)$. In either case, we have $b(x) \prec b(y)$.

Conversely, if $b(x) \prec b(y)$, then $x\not= y$. If $y > x$, then by the first part of the proof we have $b(y) \prec b(x)$ which is a contradiction. Hence, $x<y$.
\end{proof}

So, under any transformation $R_{\beta, \alpha}$ the alternate ordering respects the natural ordering on $\mathbb R$.

\subsection{Characterizing sequences}
We want to have a characterization of the sequences that are generated by a transformation $R=R_{\beta, \alpha}$. Let $\Sigma_R$ denote the set of all digit sequences generated by $R$. We use $\Delta (b_1 \cdots b_n)$ to denote the fundamental interval in $[M^-,M^+]$ specified by the digits $b_1 , \ldots, b_n$:
\[ \Delta(b_1 \cdots b_n) = \{ x \in [M^-,M^+] \, : \, b_j(x) = b_j, \, 1 \le j \le n\}.\]
Results from \cite{Hof81} by Hofbauer give that a sequence $b=b_1 b_2 \cdots \in \{0,1 \}^{\mathbb N}$ is generated by $R$ if and only if for each $n \ge 1$,
\begin{equation}\label{q:btilde}
 \text{if } b_n=1, \, \text{then } b(M^-) \preceq \ b_n b_{n+1} \cdots \prec \tilde b(\alpha), \ \text{and if } b_n =0, \, \text{then } b(\alpha) \preceq b_n b_{n+1} \cdots \preceq b(M^+),
\end{equation}
where
\[ \tilde b(\alpha) = \lim_{t \uparrow \alpha, \, t \in \Delta(1)} b(t).\]
We want to give a description of the sequence $\tilde b(\alpha)$ and therefore we define a sequence of transformations first. This sequence is obtained by alternating the transformations $L$ and $R$. First, let $L_0$ be the identity and $L_1=L$. Then, for $n \ge 1$, set $L_{2n} = (R \circ L)^n$ and $L_{2n+1} = L \circ (R \circ L)^n$. We use this sequence $\{L_n\}_{n \ge 0}$ to make digit sequences of points in $[M^-,M^+]$. For $n \ge 1$, let
\[d_{2n-1} (x) = \left\{
\begin{array}{ll}
1, &  \text{if }\ L_{2n-2} x \le \alpha,\\
0, & \text{if }\ L_{2n-2} x > \alpha,
\end{array}
\right. \ \text{and }\  d_{2n} (x) = \left\{
\begin{array}{ll}
1, & \text{if }\ L_{2n-1} x < \alpha,\\
0, &  \text{if }\ L_{2n-1} x \ge \alpha.
\end{array}
\right.
\]
Then $d(x) = d_1(x) d_2(x) \cdots$.

\begin{remark}\label{r:nodifference}
{\rm Note that for each $x$ such that $R^n x \neq \alpha$ for all $n \ge 0$, we have $R^n x =L^n x =L_n x$ for each $n$. Also for the digit sequence $d(x)$, the difference between even and odd indexed digits is only in the point $\alpha$ itself. So, for each $x$ such that $R^n x \neq \alpha$ for all $n \ge 0$, the digit sequences $b(x)$ and $d(x)$ are equal. In the above, it is crucial that $x\not= \alpha$, otherwise the remark is not true. To see this, let $\beta=\frac{1+\sqrt{5}}{2}$, and
$\alpha=-\frac{1}{\beta^2}$, then $R^n\alpha\not= \alpha$ for all $n\ge 1$, but $R^n\alpha\not= L^n\alpha =\alpha$. Also, $b(\alpha)=001010101010\cdots$, while $d(\alpha)=100101010101\cdots$.
}\end{remark}

The next lemma says that the digit sequence $d(x)$ gives negative $\beta$-expansions.

\begin{lemma}\label{l:d}
For each $x \in [M^-,M^+]$ and each $n \ge 1$, we have
\begin{equation}\label{q:d}
x = \sum_{k=1}^n (-1)^k \frac{d_k(x)}{\beta^k} + (-1)^n \frac{L_n x}{\beta^n},
\end{equation}
and thus $x = \sum_{k=1}^{\infty} (-1)^k \frac{d_k(x)}{\beta^k}$.
\end{lemma}

\begin{proof}
The lemma follows easily by observing that for each $n\ge 1$, we have $L_n x = -\beta L_{n-1} x -d_n(x)$.
\end{proof}

The next theorem gives a characterization of the digit sequences generated by $R$.
\begin{theorem}\label{t:kar}
Let $b=b_1 b_2 \cdots \in \{0,1 \}^{\mathbb N}$. Then, $b \in \Sigma_R$ if and only if for all $n \ge 1$,
\begin{equation}\label{q:kar}
\text{if } b_n=1, \, \text{then } b(M^-) \preceq \ b_n b_{n+1} \cdots \prec d(\alpha), \, \text{and if } b_n =0, \, \text{then } b(\alpha) \preceq b_n b_{n+1} \cdots \preceq b(M^+).
\end{equation}
\end{theorem}

\begin{proof}
Set $\tilde b=\tilde b(\alpha)$ and $d=d(\alpha)$. By (\ref{q:btilde}) we only need to show that $\tilde b=d$.

\vskip .1cm
First note that if $L^k \alpha \neq \alpha$ for all $k \ge 1$, then $L^k \alpha=L_k \alpha=R^{k-1}L \alpha \not= \alpha$ for all $k \ge 1$.
Hence, $\tilde b = 1b(L \alpha) = 1d(L\alpha) = d$. So, assume $L^k \alpha =\alpha$ for some $k\ge 1$, and let $n$ be the least positive integer such that
$L^n \alpha =\alpha$. Then, $L^j \alpha = R^{j-1} L \alpha = L_j \alpha \not= \alpha$ for $1 \le j \le n-1$, and $L^n \alpha = R^{n-1} L \alpha = L_n \alpha = \alpha$.
Thus, $\tilde b_j=d_j$ for all $1 \le j \le n$, and $L_n \alpha=\alpha$ is an endpoint of $R^n \Delta (\tilde b_1 \cdots \tilde b_n)$.

If $n$ is even, then
$\alpha$ is a right end-point of $R^n \Delta (\tilde b_1 \cdots \tilde b_n)$
so that
$d_{n+1}=1$. Also, for all $x \in \Delta(\tilde b_1 \cdots \tilde b_{n+1})$ we have $\frac{dR^n x}{dx} = \beta^n$, so $R^n x< L_n \alpha = \alpha$. Since $R^n \Delta(\tilde b_1 \cdots \tilde b_{n+1}) \subseteq \Delta(\tilde b_{n+1})$, this implies $\tilde b_{n+1}=1 =d_{n+1}$. Since $L_{n+1} \alpha = (L \circ L_n) \alpha$, $L_{n+1} \alpha$ is an endpoint of the interval $R^{n+1} \Delta (\tilde b_1 \cdots \tilde b_{n+1})$.

On the other hand, if $n$ is odd, then $\alpha$ is a left end-point of $R^n \Delta (\tilde b_1 \cdots \tilde b_n)$ so that
$d_{n+1}=0$. Then, for all $x \in \Delta(\tilde b_1 \cdots \tilde b_{n+1})$, we have $\frac{dR^n x}{dx} = -\beta^n$ and thus $R^n x > L_n \alpha = \alpha$ and $\tilde b_{n+1}= 0 =d_{n+1}$. Now $L_{n+1} \alpha = (R \circ L_n) \alpha$, so also here $L_{n+1} \alpha$ is an endpoint of $R^{n+1} \Delta (\tilde b_1 \cdots \tilde b_{n+1})$.

The same reasoning holds when $L_k \alpha = \alpha$ for a $k >n$, so this gives the theorem.
\end{proof}

\subsection{What is greedy?}
For expansions with a positive non-integer base, there is a well-understood notion of greedy $\beta$-expansions. For numbers that have more than one $\beta$-expansion, the greedy $\beta$-expansion is the one that has the largest digit sequence in the lexicographical ordering. These expansions are the ones that are produced by the map
\[ x \mapsto \left\{
\begin{array}{ll}
\beta x  \, (\text{mod } 1), &  \text{if } x \in [0,1),\\
\\
\beta x - \lfloor \beta \rfloor, & \text{if } x \in \big[ 1 , \frac{\lfloor \beta \rfloor}{\beta -1} \big].
\end{array}
\right.\]
A natural candidate for the negative greedy $\beta$-expansion, would be the one that is largest in the alternate ordering.
\begin{defn}[Greedy expansion]
Let $1 < \beta < 2$. Let $x \in [M^-,M^+]$ have the negative $\beta$-expansion
$ x = \sum_{k=1}^{\infty} (-1)^k \frac{b_k}{\beta^k}$,
with $b_k \in \{0,1\}$ for all $k \ge 1$. Set $b=b_1b_2 \cdots$. Then this expansion is the {\em negative greedy $\beta$-expansion of $x$ with digits in $\{0,1\}$} if for each sequence $d=d_1 d_2 \cdots \in \{ 0,1\}^{\mathbb N}$, such that
$ x = \sum_{k=1}^{\infty} (-1)^k \frac{d_k}{\beta^k}$,
we have $d \preceq b$.
\end{defn}

The next proposition shows that there is no transformation $R_{\beta, \alpha}$ that generates the negative greedy $\beta$-expansion of $x$ with digits in $\{0,1\}$ for all $x \in [M^-,M^+]$.
\begin{prop}
Let $1< \beta <2$. Then there is no $\alpha \in S$, such that for all $x \in [M^-,M^+]$ the digit sequence for $x$ generated by $R_{\beta, \alpha}$ gives the greedy expansion of $x$.
\end{prop}

\begin{proof}
Note that if $x \in S$, then by Lemma~\ref{l:firstdigit} $b_1(x)$ can be either $0$ or $1$. Since we want to get greedy expansions, for each $x \in S$, we need $b_1(x) =0$. Hence, on $S$, we define $R_{\beta , \alpha} x = -\beta x$. This means that $\alpha = \frac{1}{\beta^2-1}-\frac{1}{\beta}$. Now, consider the interval
\[ I=\Big[ \frac{1}{\beta^2} - \frac{1}{\beta(\beta^2-1)}, \frac{1}{\beta^2(\beta^2-1)}\Big] \subseteq  R^{-1}_{\beta, \alpha} S \cap U_0.\]
Then, for each $x \in I$, $b_1(x) =0$ and $R_{\beta, \alpha} x \in S$. To get the greedy expansion for elements $x\in I$, we need to assign the digit $b_2(x) = b_1(R_{\beta, \alpha} x) =1$, which contradicts the previous choice of $\alpha$. Hence, there is no transformation $R_{\beta, \alpha}$ that generates the greedy expansion for all $x \in [M^-,M^+]$.
\end{proof}

Among the family of transformations $\{R_{\beta,\alpha}:\alpha \in S_{\beta}\}$, one can speak of the {\em odd greedy transformation} obtained by choosing
$\alpha = -\frac{1}{\beta(\beta^2-1)}$. Note that if $x$ has two negative $\beta$-expansions with different first digit, i.e., $x = \sum_{k=1}^{\infty}(-1)^k \frac{b_k}{\beta^k} = \sum_{k=1}^{\infty}(-1)^k \frac{d_k}{\beta^k}$ with $b_1=0$ and $d_1 =1$, then $ d_1 d_2 \cdots  \prec b_1 b_2\cdots $ and this choice of $\alpha$ would give $b_1=0$. The next proposition gives a recursive algorithm to obtain the digit sequences of the odd greedy transformation. Let $\ell = \overline{01}$ be the largest sequence in alternate ordering.
\begin{prop}
Let $1< \beta < 2$ and $\alpha = -\frac{1}{\beta(\beta^2-1)}$. Let $b_1 b_2 \cdots \in \{0,1\}^{\mathbb N}$ and $x = \sum_{k=1}^{\infty}(-1)^k\frac{b_k}{\beta^k}$. Then $b_1 b_2 \cdots$ is the digit sequence of $x$ generated by $R=R_{\beta, \alpha}$ if it satisfies the following recursive conditions. Suppose $b_1, b_2, \ldots, b_{n-1}$ are known. If $n$ is odd, then $b_n$ is the smallest element of $\{0,1\}$ such that
\[ \sum_{k=1}^{n-1} (-1)^k \frac{b_k}{\beta^k} - \frac{b_n}{\beta^n} + (-1)^n\frac{1}{\beta^n}\sum_{k=1}^{\infty} (-1)^k \frac{\ell_k}{\beta^k} \le x.\]
If $n$ is even, then $b_n$ is the smallest element of $\{0,1\}$ such that
\[ \sum_{k=1}^{n-1} (-1)^k \frac{b_k}{\beta^k} + \frac{b_n}{\beta^n} + (-1)^n\frac{1}{\beta^n}\sum_{k=1}^{\infty} (-1)^k \frac{\ell_k}{\beta^k} \ge x.\]
\end{prop}

\begin{proof}
Assume that the sequence $b_1 b_2 \cdots$ satisfies the hypothesis. We want to show that $b_1 b_2$ gives the expansion of $x$ that is generated by $R$, i.e., that $b_n =1$ if $R^{n-1} x < -\frac{1}{\beta(\beta^2-1)}$ and $b_n =0$ if $R^{n-1} x \ge -\frac{1}{\beta (\beta^2-1)}$. It is enough to prove the proposition for $n=1,2$.

\vskip .1cm

Suppose that $b_1=0$. Then, by the hypothesis,
\[ -\frac{0}{\beta} + \frac{0}{\beta^2}-\frac{1}{\beta^3}+\frac{0}{\beta^4} - \cdots = -\frac{1}{\beta (\beta^2-1)} \le x.\]
If $b_1 =1$, then
\[ x < -\frac{0}{\beta} + \frac{0}{\beta^2}-\frac{1}{\beta^3}+\frac{0}{\beta^4} - \cdots = -\frac{1}{\beta (\beta^2-1)}.\]
This shows that in both cases $b_1$ is the digit generated by $R$ and $x = -\frac{b_1}{\beta}-\frac{Rx}{\beta}$. For $n=2$, if $b_2=0$, then
\[ -\frac{b_1}{\beta} + \frac{0}{\beta^2} -\frac{0}{\beta^3} + \frac{1}{\beta^4} -\cdots \ge x=-\frac{b_1}{\beta}-\frac{Rx}{\beta}.\]
Hence, $-\frac{Rx}{\beta} \le \frac{1}{\beta^2(\beta^2-1)}$ and thus $Rx \ge -\frac{1}{\beta(\beta^2-1)}$. If $b_2=1$, then
\[ x=-\frac{b_1}{\beta}-\frac{Rx}{\beta} > -\frac{b_1}{\beta}  + \frac{0}{\beta^2} -\frac{0}{\beta^3} + \frac{1}{\beta^4} -\cdots.\]
Thus, $Rx > -\frac{1}{\beta(\beta^2-1)}$. Again, we see that $b_2$ is the digit generated by $R$. This gives the result.
\end{proof}

\section{The number of negative $\beta$-expansions}\label{s:random}
\subsection{Switch regions and infinitely many expansions}
For all $1< \beta < 2$, we can divide the interval $[M^-,M^+]$ into the switch region $S$ and the uniqueness regions $U_0$ and $U_1$, see (\ref{q:switch}). Then, we can define a random transformation, as was done in \cite{DK03} and \cite{DV05}. Let  $\Omega = \{ 0,1\}^{\mathbb{N}}$ endowed with the product $\sigma$-algebra $\mathcal{F}.$ Let $\sigma : \Omega \to \Omega$ be the left shift, and define $K_{\beta} : \Omega \times [M^-,M^+] \to \Omega \times [M^-,M^+]$ by
\begin{equation*}
K_{\beta}(\omega , x)\, =\, \left\{ \begin{array}{ll}
(\omega ,- \beta x - j ), &\text{if } x\in U_j,\; j \in \{0,1\} ,\\
 & \\
(\sigma (\omega ),- \beta x - \omega_1), & \text{if } x\in S.
\end{array}\right.
\end{equation*}

The elements of $\Omega$ represent the coin tosses (`heads'=1 and `tails'=0) used every time the orbit hits a switch region. Let
\[ d_1\, =\, d_1(\omega , x)\, =\, \left\{ \begin{array}{ll}
1, & {\mbox{ if }}\; x\in U_1 {\mbox{ or }}\; (\omega ,x)\in \{ \omega_1=1\} \times S,\\
  & \\
0, & {\mbox{ if }}\; x\in U_0 {\mbox{ or }}\; (\omega ,x)\in \{ \omega_1=0\} \times S,
\end{array}\right.\]
then
\[
K_{\beta}(\omega ,x)\, =\, \left\{ \begin{array}{ll}
(\omega ,- \beta x - d_1), & {\mbox{ if }}\; x\in U_0 \cup U_1,\\
 & \\
(\sigma (\omega ), -\beta x - d_1), & {\mbox{ if }}\; x\in  S.
\end{array}\right.
\]
Set $d_n=d_n(\omega ,x)=d_1\big( K_{\beta}^{n-1}(\omega ,x)\big)$,
and let $\pi_2 :  \Omega \times [M^-,M^+] \to [M^-,M^+]$ be the canonical projection onto the second coordinate. Then
\[ \pi_2 \left( K_{\beta}^n(\omega ,x)\right) \, =\,
(-1)^n\beta^nx+(-1)^n\beta^{n-1}d_1+(-1)^{n-1}\beta^{n-2}d_2+
\cdots +(-1)^2\beta d_{n-1} +(-1)^1 d_n,\]
and rewriting yields
\[ x\, =\, -\frac{d_1}{\beta} + \frac{d_2}{\beta^2} + \cdots + (-1)^n\frac{d_n}{\beta^n} +(-1)^n \frac{\pi_2 \big( K_{\beta}^n(\omega ,x)\big)}{\beta^n} .\]
Since $\pi_2 \big( K_{\beta}^n(\omega ,x)\big) \in [M^-,M^+]$, it follows that
\[ \Big| x - \sum_{k=1}^n (-1)^k\frac{d_k}{\beta^k} \Big| \, =\,
\frac{\pi_2 \big( K_{\beta}^n(\omega ,x)\big)}{\beta^n} \, \to \, 0
\qquad {\mbox{as }}\; n\to \infty .\]
This shows that for all $\omega \in \Omega$ and for all $x\in [M^-,M^+]$  one has that
\[ x\, =\, \sum_{k=1}^{\infty}(-1)^k \frac{d_k}{\beta^k}\, =\, \sum_{k=1}^{\infty}(-1)^k \frac{d_k(\omega ,x)}{\beta^k}.\]
The random procedure just described shows that with each $\omega \in \Omega$ corresponds an algorithm that produces expansions in base $\beta$. If we identify the point $(\omega ,x)$ with $(\omega, d_1(\omega ,x) d_2(\omega ,x)\cdots )$, then the action of
$K_{\beta}$ on the second coordinate corresponds to the left shift. We call the sequence $d_1(\omega ,x) d_2(\omega ,x)\cdots $ the random negative $\beta$-expansion of $x$ specified by $\omega$.

One can easily generalize the proof of Theorem 2 in \cite{DV05} to obtain the following theorem.
\begin{theorem} Let $x\in [M^-,M^+]$, and let $x=\sum_{n=1}^{\infty}(-1)^n \frac{b_n}{\beta^n}$ with $b_n\in \{0,1 \}$ be a representation of $x$ in base $-\beta$. Then there exists an $\omega \in \Omega $ such that $b_n=d_n(\omega,x).$
\end{theorem}

Using the map $K_{\beta}$, one can generate greedy expansions in base $-\beta$, i.e., expansions that are the largest in the alternate ordering.
We have the following theorem.

\begin{theorem} Let $x\in [M^-,M^+]$, then there exists $\omega \in \Omega$ such that $d_1(\omega ,x) d_2(\omega ,x)\cdots $, the random negative $\beta$-expansion of $x$ specified by $\omega$, is the greedy expansion of $x$.
\end{theorem}

\begin{proof} Let $x\in [M^-,M^+]$, and set $x_0=x$.  We define inductively a sequence of cylinders $\Omega_1 \supseteq \Omega_2\supseteq \cdots$ as follows.
\begin{itemize}
\item If $x_0\in U_j$ for $j\in \{0,1\}$, then set $x_1=-\beta x-j$, $\ell_1(x)=0$ and $\Omega_1=\Omega$.
\item If $x_0\in S$, then set $x_1=-\beta x$, $\ell_1(x)=1$ and $\Omega_1=\{\omega\in \Omega: \omega_1=0\}$.
\end{itemize}
We now consider $x_1$.
\begin{itemize}
\item If $x_1\in U_j$ for $j\in \{0,1\}$, then set $x_2=-\beta x-j$, $\ell_2(x)=\ell_1(x)$ and $\Omega_2=\Omega_1$.
\item If $x_1\in S$, then set $x_2=-\beta x-1$, $\ell_2(x)=\ell_1(x)+1$ and $\Omega_2=\{\omega\in \Omega_1: \omega_{\ell_2(x)}=1\}$.
\end{itemize}
Suppose that $\{x_1,\cdots, x_n\}$, $\{\ell_1(x),\cdots,\ell_n(x)\}$ and  $\Omega_1 \supseteq \Omega_2\supseteq \cdots \supseteq \Omega_n$ have been defined.

\noindent
{\bf Case 1:} Assume $n$ is even.
\begin{itemize}
\item If $x_n\in U_j$ for $j\in\{0,1\}$, then set $x_{n+1}=-\beta x-j$, $\ell_{n+1}(x)=\ell_{n}(x)$ and $\Omega_{n+1}=\Omega_n$.
\item If $x_n\in S$, then set $x_{n+1}=-\beta x$, $\ell_{n+1}(x)=\ell_{n}(x)+1$ and $\Omega_{n+1}=\{\omega\in \Omega_n: \omega_{\ell_{n+1}(x)}=0\}$.
\end{itemize}
{\bf Case 2:} Assume $n$ is odd.
\begin{itemize}
\item If $x_n\in U_j$ for $j \in \{0,1\}$, then set $x_{n+1}=-\beta x-j$, $\ell_{n+1}(x)=\ell_{n}(x)$ and $\Omega_{n+1}=\Omega_n$.
\item If $x_n\in S$, then set $x_{n+1}=-\beta x-1$, $\ell_{n+1}(x)=\ell_{n}(x)+1$ and $\Omega_{n+1}=\{\omega\in \Omega_n: \omega_{\ell_{n+1}(x)}=1\}$.

\end{itemize}

If $K_{\beta}$ hits the switch regions infinitely many times, then $\ell_{n}(x) \rightarrow
\infty$ and, as is well known, $\bigcap \Omega_{n}$ consists of a single point. If this happens only finitely many times, then the set $\lbrace \ell_{n}(x) \, : \, n \in \mathbb{N} \rbrace$ is finite and
$\bigcap \Omega_{n}$ is exactly a cylinder set. In both cases $\bigcap \Omega_{n}$ is non-empty and for any $\omega \in \bigcap \Omega_{n}$, the random negative $\beta$-expansion
of $x$ specified by $\omega$, is the greedy expansion of $x$.
\end{proof}

\subsection{Uniqueness regions and unique expansions}

\begin{prop}
The set of $x\in [M^-,M^+]$ that has a unique negative $\beta$-expansion with digits in $\{0,1\}$ has Lebesgue measure zero. Moreover, if $\beta < \frac{1+\sqrt 5}{2}$, then $M^-$ and $M^+$ are the only two points with a unique negative $\beta$-expansion.
\end{prop}

\begin{proof}
Recall from (\ref{q:switch}) that $[M^-,M^+] = U_0 \cup S \cup U_1$. A point $x \in [M^-,M^+]$ as a unique negative $\beta$-expansion if and only if for each choice of $\alpha \in S$ and for all $k \ge 0$, $R^k_{\beta, \alpha} x  \in U_0 \cup U_1$.  Fix $\alpha \in S$, set $R=R_{\beta, \alpha}$ and let $\mu$ be the unique, ergodic acim for $R$. The support of $\mu$ contains an interval with $\alpha$ in its interior. Let $C$ denote the support of $\mu$, then $\mu(C \cap S)>0$. Let $B$ be the set of points in $[M^-,M^+]$ with a unique negative $\beta$-expansion with digits in $\{0,1\}$. Suppose that $\mu(B) >0$. By the ergodicity of $R$, there is a $k$, such that $\mu(B \cap R^{-k}(C \cap S))>0$, which gives a contradiction. Hence, $\mu(B)=0$. Since $\lambda$ and $\mu$ are equivalent on $C$, this implies that $\lambda(B\cap C)=0$, i.e., $\lambda$-a.e.~$x \in C$ has more than one expansion.

If $\beta \le \frac{1+\sqrt 5}{2}$, then $\beta^2-\beta-1 \le 0$ and thus
\[ \frac{1}{\beta^2-1} - \frac{1}{\beta}  \ge 0, \quad \text{ and } \quad -\frac{1}{\beta(\beta^2-1)} \le \frac{1}{\beta+1}.\]
Then $R U_0 \subseteq U_1 \cup S$ and $R U_1 \subseteq U_0 \cup S$. This implies that for each $x \in (M^-,M^+)$ there is a $k=k(x)$, such that $R^k x \in S$. Hence, the only points with a unique expansion are $M^-$ and $M^+$. This gives the second part of the proposition.
\end{proof}

\begin{remark}
{\rm Everything in this article except Sections 2.2 and 4.2 can be extended to more digits. In general, a class of transformations that generate negative $\beta$ expansions can be given for each combination of $\beta>1$ and set of real numbers $A=\{ a_0, \ldots, a_m\}$ that satisfy:
\begin{itemize}
\item $a_0 < a_1 < \cdots < a_m$,
\item $\displaystyle \max_{1 \le j \le m} (a_j-a_{j-1}) \le \frac{a_m-a_0}{\beta-1}$.
\end{itemize}
These transformations are given by choosing an $\alpha$ for each pair of digits $a_j, a_{j+1}$ and thus have $m$ points of discontiuity. Results from \cite{LasY73} imply that each of these transformations has an acim. The previously mentioned results from \cite{LY78} give that the number of ergodic components is at most $m$ and that the support of each acim is a forward invariant set, containing at least one of the points of discontinuity in its interior. To find the density, we can use the same trick from \cite{Hof81}. Also, the set of digit sequences is characterized in exactly the same way as for two digits, with a condition for each digit. To find a transformation that generates greedy expansions, we have to turn to a random transformation also here. This map can be constructed similarly to as was done in \cite{DK07} for $\beta$-expansions with arbitrary digits.}
\end{remark}

\bibliographystyle{alpha}
\bibliography{negbeta}
\end{document}